\documentclass{amsart}
\usepackage{amssymb,amsthm}
\usepackage{hyperref}

\usepackage{xcolor}

\newcommand{\C}{{\mathbb{C}}}
\newcommand{\I}{{\mathfrak{I}}}
\newcommand{\R}{{\mathbb{R}}}

\newcommand{\Z}{{\mathbb{Z}}}
\newcommand{\op}{\textit{op}}

\newcommand{\HH}{\mathsf{H}} 
\newcommand{\TT}{\mathsf{T}} 

\newcommand{\qtq}[1]{\quad\text{#1}\quad}
\newcommand{\eps}{\varepsilon}

\DeclareMathOperator{\sgn}{sgn}
\DeclareMathOperator{\tr}{tr}
\DeclareMathOperator{\dist}{dist}

\DeclareMathOperator{\supp}{supp}
\DeclareMathOperator{\Log}{Log}

\let\Re=\undefined\DeclareMathOperator{\Re}{Re}
\let\Im=\undefined\DeclareMathOperator{\Im}{Im}
\let\det=\undefined\DeclareMathOperator{\det}{det} 

\newtheorem{theorem}{Theorem}[section]
\newtheorem{prop}[theorem]{Proposition}
\newtheorem{lemma}[theorem]{Lemma}
\newtheorem{corollary}[theorem]{Corollary}
\theoremstyle{definition}

\numberwithin{equation}{section}

\allowdisplaybreaks

\begin{document}

\title[A Jost--Pais theorem for the intermediate long wave equation]{A Jost--Pais theorem for the\\intermediate long wave equation}
\author{Rowan Killip, Monica Vi\c{s}an, and Yilun Wu}

\address
{Rowan Killip\\
CEREMADE, Universit\'e Paris Dauphine--PSL, 75016 Paris, France}
\email{killip@ceremade.dauphine.fr}

\address
{Monica Vi\c{s}an\\
Institute of Science and Technology Austria, Am Campus 1, 3400 Klosterneuburg, Austria}
\email{Monica.Visan@ist.ac.at}

\address
{Yilun Wu\\
Department of Mathematics, 
University of Oklahoma, OK 73019, USA}
\email{allenwu@ou.edu}

\begin{abstract}
For a one-dimensional Schr\"odinger operator, the transmission coefficient can be expressed as a Fredholm determinant.  This fundamental discovery of Jost and Pais has recently proven very influential in the study of conserved quantities for the Korteweg--de Vries equation, whose Lax operator takes this form. In this paper, we obtain an analogue of this result for the intermediate long wave equation, where, startlingly, the transmission coefficient is given by a \emph{ratio} of Fredholm determinants.  We then recover the original Jost--Pais result in the shallow-water limit and its analogue for the Benjamin--Ono equation as the infinite-depth limit. 
\end{abstract}

\maketitle

\section{Introduction}

The intermediate long wave equation describes the propagation of internal waves at the interface of a stratified fluid of total depth $h>0$; see \cite{MR0455822}.  It takes the form
\begin{align}\tag{ILW}\label{ILW}
\tfrac{d}{dt}q = \TT q'' - \tfrac1h q' -2qq',
\end{align}
with $\TT$ denoting the Fourier multiplier operator\footnote{We have reversed the sign of this ``Tilbert transform'' relative to many works; we felt compelled to do so to match the mathematical convention for the Hilbert transform \eqref{H defn}.} 
\begin{equation}\label{T defn}
\widehat {\TT f}(\xi) := - i\coth(h\xi) \widehat f(\xi).
\end{equation}
The unknown real-valued function $q(t,x)$ represents the elevation of the interface relative to equilibrium at time $t\in\R$ and spatial position $x\in\R$.

This equation is termed the \emph{intermediate} long wave equation because it interpolates between the Korteweg--de Vries equation
\begin{align}\tag{KdV}\label{KdV}
\tfrac{d}{dt}q = -q''' - 6qq' ,
\end{align}
which is used to model both surface and interfacial waves in shallow water, and the Benjamin--Ono equation
\begin{align}\tag{BO}\label{BO}
\tfrac{d}{dt}q = \HH q'' -2qq' ,
\end{align}
which describes interfacial waves in the deep water ($h\to \infty$) limit.  Here, $\HH$ denotes the Hilbert transform
\begin{equation}\label{H defn}
\widehat {\HH f}(\xi) := -i\sgn(\xi) \widehat f(\xi).
\end{equation}

Remarkably, \eqref{ILW} is completely integrable!  While \eqref{KdV} and \eqref{BO} have long been known to be completely integrable, their underlying algebraic and analytic structures are rather disparate and so it is striking that the chasm between them can be bridged by a one-parameter family of integrable models.  For this reason, \eqref{ILW} bears the important role of unifier, in addition to its direct physical significance.

This paper grew out of the consideration of the conserved quantities associated to \eqref{ILW}.  To give context, let us first discuss two different approaches to  conservation laws for \eqref{KdV} and how they are unified by the original Jost--Pais Theorem.
The first approach is through the scattering theory of the Lax operator.  In the \eqref{KdV} setting, this is a one-dimensional Schr\"odinger operator:
\begin{equation}\label{Lax:KdV}
L_q^\textsc{kdv}:= D^2  - q \qtq{where} D:=-i\partial_x.
\end{equation}

When $q\in L^1(\R)$, ODE methods guarantee the existence of Jost solutions, that is, formal eigenfunctions of $L_q^\textsc{kdv}$ with plane-wave asymptotics in one spatial direction.  Concretely, given a wave number $\zeta\neq 0$ with $\Im \zeta \geq 0$, there are unique solutions $\psi_\pm(x;\zeta)$ to
\begin{equation}\label{Jost:KdV}
-\psi'' - q\psi = \zeta^2 \psi \qtq{with} e^{\mp i\zeta x}\psi_\pm(x) \to 1 \quad\text{as $x\to\pm \infty$.}
\end{equation}
The key quantity for our current discussion is built from the Wronskian of these two solutions:
\begin{equation}\label{Jost:Wron}
a^\textsc{kdv}_\textup{scatt}(q;\zeta) :=  \tfrac{1}{2i\zeta}\bigl[ \psi_+'(x;\zeta)\psi_-(x;\zeta) - \psi_+(x;\zeta)\psi_-'(x;\zeta) \bigr].
\end{equation}

Using the Lax pair, one may show that $a^\textsc{kdv}_\textup{scatt}(q;\zeta)$ is conserved (in time) when $q(t)$ is (for example) a Schwartz-class solution to \eqref{KdV}.  This quantity forms the basis of the inverse scattering approach to the conservation laws; see \cite{Miura}.  As we will explain in due course, $a^\textsc{kdv}_\textup{scatt}(q;\zeta)$ represents the \emph{reciprocal} of the transmission coefficient.

A seemingly completely independent approach to the conservation laws for integrable systems is based on the perturbation determinant.  The thesis here is that the Lax pair equation $\partial_t L = [P,L]$ ensures that $L(t)$ and $L(0)$ are conjugate.  In the case of finite matrices, this would guarantee equality of the characteristic polynomials.  For a Lax operator such as \eqref{Lax:KdV}, a renormalization is necessary. Concretely, we consider the perturbation determinant relative to the operator $L_0^\textsc{kdv}$ that corresponds to the static $q\equiv 0$ solution:
\begin{equation}\label{Jost:pdet}
a^\textsc{kdv}_\textup{pdet}(q;\zeta) := \det\bigl[ (L_0-\zeta^2)^{-1} (L_q-\zeta^2) \bigr] = \det\bigl[ 1 - (L_0-\zeta^2)^{-1} q\bigr]  .
\end{equation}
The determinant here may be understood either through the classical theory of Fredholm \cite{Fredholm} or via the more modern theory of trace-class operators on Hilbert spaces \cite{MR2154153}, each requiring slightly different hypotheses; both approaches easily cover Schwartz $q$.  

One may show by direct computation that $a^\textsc{kdv}_\text{pdet}(q;\zeta)$ is likewise conserved for Schwartz-class solutions to \eqref{KdV}; see\cite{KVZ}.  This is no coincidence:

\begin{theorem}[Jost--Pais \cite{JP51}]\label{T:JP}
Suppose $\Im \zeta >0$ and $q$ is Schwartz.  Then
\begin{equation}\label{E:JP}
a^\textsc{kdv}_\textup{scatt}(q;\zeta) = a^\textsc{kdv}_\textup{pdet}(q;\zeta) ,
\end{equation}
which is to say, the transmission coefficient is the reciprocal of the perturbation determinant.
\end{theorem}

Strictly speaking, the original paper \cite{JP51} treats scattering for radial potentials in three dimensions; however, their argument carries over easily to the whole-line setting.  An alternate proof can be found in \cite[Ch.~5]{MR2154153}.

Note that the conservation of \eqref{E:JP} under \eqref{KdV} encodes the original infinite sequence of conserved quantities described in \cite{GGKMII}; these are the coefficients of the asymptotic series for $\log[a^\textsc{kdv}(q;\zeta)]$ as $\zeta\to i\infty$.  The existence of such an asymptotic series is one of the reasons we elected to focus our attention to Schwartz potentials $q$.  A second reason is that \eqref{KdV} is wellposed in this space, making it an excellent context in which to discuss conserved quantities.

In truth, nothing significant is lost if one proves the equality \eqref{E:JP} only for $q\in C^\infty_c(\R)$; this class is dense in all spaces of interest and there is ample technology to prove that both sides are continuous functions of $q$.  One advantage of demanding that $q$ has compact support is that the asymptotic behaviour of the Jost solutions $\psi_\pm$ becomes trivial; indeed, $\psi_\pm=e^{\pm i\zeta x}$ on one side of the potential, while on the other side, it is a linear combination of $e^{i\zeta x}$ and $e^{-i\zeta x}$.  Recognizing the propagation of such plane waves to the left and right then makes it apparent why we refer to \eqref{Jost:Wron} as the reciprocal of the transmission coefficient. 

The study of the conserved quantities via each side of \eqref{E:JP} has its advantages and its adherents.  A deep understanding of Jost solutions is a prerequisite for developing inverse scattering theory, which has proved to be the most potent tool for understanding the long time behaviour of solutions.  On the other hand, the perturbation determinant has proven to be an extremely effective way to understand conservation laws for solutions with slow decay and/or low regularity.  Moreover, it provides an efficient means to implement renormalization, which is often required in such regimes.  Lastly, this approach also carries over seamlessly to the study of spatially periodic solutions.

A rigorous treatment of the inverse scattering approach to \eqref{ILW} has recently been under development by Perry and Wu \cite{PerryWu}, building on the framework proposed in \cite{MR0591811}.   Independently, conservation laws and resulting $H^s$-bounds were derived in \cite{HGKV:ILW} through the analysis of the renormalized perturbation determinant.  Each party employed a different Lax pair; these competing choices are reconciled in \cite[Appendix~B]{PerryWu}.  In this paper, we will adopt the choice
\begin{equation}\label{LP}\begin{aligned}
L_q&:= D + \tfrac1{2h}\bigl(e^{-2hD}-1\bigr) - q \qtq{again with} D:=-i\partial_x, \\
P_q &:= -\tfrac ih D + i D^2  + \tfrac ihD e^{-2hD} + i(\TT q') - i (Dq +q D)
\end{aligned}\end{equation}
from \cite{HGKV:ILW}, because it behaves slightly better under variations of $h$.  This will be useful when we consider the limiting cases $h\to0$ and $h\to\infty$ toward the end of the paper.

The different choices of Lax pair in \cite{HGKV:ILW,PerryWu} do not explain a fundamental difference that we discovered between the two approaches: While the logarithm of the transmission coefficient does serve as a generating function for the traditional conserved quantities, the logarithm of the renormalized perturbation determinant does not.  Evidently, the relation between the two approaches is more complicated than \eqref{E:JP}!

In order to state the \eqref{ILW} analogue of the Jost--Pais Theorem, we must first define a transmission coefficient.  This is delicate.  In the case $q\equiv 0$, we notice that the relation between the spectral parameter $z\in\C$ and the wave number $\zeta\in\C$ is given by
\begin{equation}\label{f defn}
z = f_h(\zeta) \qtq{with} f_h(\zeta) := \zeta + \tfrac1{2h}\bigl(e^{-2h\zeta}-1 \bigr).
\end{equation}
As we will see later, this means that for each $z\in\C$, there are \emph{infinitely many} $\zeta\in\C$ so that $L_0 e^{i\zeta x} = z e^{i\zeta x}$.  This is a far cry from the simple left and right propagating waves of the \eqref{KdV} case!   Secondly, the Lax operator $L_q$ is highly nonlocal and we cannot construct Jost-like solutions in the same simple manner.

In order to arrive quickly at our main result, we make several bold assertions that will be clarified in Section~\ref{S:2}. 
Consider the two strips 
\begin{equation}\label{om pm}
\Omega_h^+= \bigl\{\zeta\in \C:\, 0<\Im \zeta<\tfrac\pi h\bigr\} \qtq{and} \Omega_h^-= \bigl\{\zeta\in \C:\, - \tfrac\pi h<\Im \zeta< 0\bigr\}.
\end{equation}
The mapping $f_h$ defined in \eqref{f defn} carries the closure of each of these strips \emph{onto} the whole complex plane.  Moreover, if $q\in C^\infty_c(\R)$ and
\begin{equation}\label{phi hyp}
\zeta\in \Omega_h^\pm \qtq{with} f_h(\zeta)\notin \sigma(L_q),
\end{equation}
then there is a unique solution $\phi_\pm=\phi_\pm(x; \zeta,q)$ to
\begin{equation}\label{phi pm}
\phi_\pm = e^{i\zeta x} - R_0(\zeta) q \phi_\pm \qtq{with} \phi_\pm - e^{i\zeta x} \in L^2(\R). 
\end{equation}
Here $R_0(\zeta)$ denotes the resolvent of the free (i.e. $q\equiv 0$) Lax operator:
\begin{equation}\label{R0 defn}
R_0(\zeta) := \bigl[ L_0 - f_h(\zeta)\bigr]^{-1} = \bigl[ f_h(D) - f_h(\zeta)\bigr]^{-1}.
\end{equation}

The transmission coefficient is then defined for $\zeta$ satisfying \eqref{phi hyp} as follows:
\begin{equation}\label{t from phi}
\mathfrak t (\zeta) := \lim_{x\to \pm\infty} e^{-ix\zeta} \phi_\pm(x; \zeta,q) \qtq{for} \zeta\in\Omega_h^\pm \qtq{with} f_h(\zeta)\notin \sigma(L_q).
\end{equation}
The existence of this limit will be discussed in Section~\ref{S:2}. There we will see that $\phi_\pm$ are normalized to correspond to an incoming wave of unit amplitude, while the Jost solutions $\psi_\pm$ of \eqref{Jost:KdV} were normalized to have a transmitted wave of unit amplitude.  This explains why \eqref{t from phi} is called the transmission coefficient, while \eqref{Jost:Wron} is the \emph{reciprocal} of the transmission coefficient.  

It turns out there is also a difficulty in defining the perturbation determinant in the \eqref{ILW} context: $q\in C^\infty_c(\R)$ does not imply that the operator $R_0(\zeta) q$ is trace class!  We will need to employ the Hilbert-regularized determinant:
\begin{equation}
\det_2(1+A) := \det(1+A)e^{-\tr A} ,
\end{equation}
which extends Fredholm theory to Hilbert--Schmidt class; see \cite{Hilbert, MR2154153}.

The main result of this paper is the following \eqref{ILW} analogue of the classical Jost--Pais identity:

\begin{theorem}\label{T:main}
For any $q\in C^\infty_c(\R)$ and $\zeta\in\C$ satisfing \eqref{phi hyp}, the transmission coefficient defined in \eqref{t from phi} can be written as
\begin{equation}\label{E:main}
\mathfrak t(\zeta;q) = 
\frac{\det_2\bigl(1 - R_0(\zeta\mp\frac{i\pi}h) q\bigr)}{\det_2\bigl(1- R_0(\zeta) q\bigr)} \exp\biggl\{\pm \tfrac{i}{1-e^{-2h\zeta}}\int_\R q(x)\, dx \biggr\}.
\end{equation}
\end{theorem}

From the point of view of conservation laws, the final term in \eqref{E:main} is acceptable; one easily verifies that $\int q$ is conserved under \eqref{ILW}.  However, we would argue that its presence represents both a shortcoming of the physical transmission coefficient and its remedy: As shown in \cite{HGKV:ILW}, the renormalized perturbation determinant extends continuously to potentials $q$ that are both very rough and slowly decaying.  Correspondingly, the final term in \eqref{E:main} indicates exactly how the transmission coefficient ought to be renormalized in order to inherit such continuity.

Theorem~\ref{T:main} will be proved in Section~\ref{S:3}, building on the prerequisites developed in Section~\ref{S:2}.  In particular, Lemma~\ref{L:Rel HS} shows that $R_0 q$ is indeed Hilbert--Schmidt, Proposition~\ref{P:phi} shows the existence of the solutions $\phi_\pm$, and Proposition~\ref{P:limit} shows existence of the limit \eqref{t from phi}.  Lemma~\ref{L:two R0} also appears in Section~\ref{S:2}, where it is used to prove Proposition~\ref{P:limit}; however, its real significance awaits in Section~\ref{S:3}, where we see that it is the crucial equation-specific fact that underlies the proof of Theorem~\ref{T:main}.

The remaining two sections of the paper are devoted to the deep and shallow-water limits, which may be modeled by \eqref{BO} and \eqref{KdV}, respectively.  There have been a number of studies analyzing how \eqref{ILW} behaves in these limits; see \cite{ABFS,CGLO,CLO,CLOP,MR4866612,HGKV:ILW,MR4752554,SAF}.  We are interested in the behavior of each side of \eqref{E:main} as $h\to \infty$ and as $h\to 0$.

The Benjamin--Ono equation was introduced in \cite{Benj,DA:BO} precisely to model interfacial waves in deep water and was later discovered to be completely integrable.  We will employ the Lax pair
\begin{equation}\label{Lax:BO}
L_q^\textsc{bo} := D - C_+ q,
	\quad P_q^\textsc{bo} := iD^2 - 2iD C_+q + 2 q'_+,  \qtq{with} q_+(t) := C_+ q(t) 
\end{equation}
from \cite{MR3484397} rather than the one that was first discovered in \cite{BK,Nakamura}.  Each operator acts on the Hardy space $L^2_+(\R)=\{f\in L^2 : \supp(\widehat f)\subseteq[0,\infty)\}$ and $C_+$ denotes the orthogonal projection of $L^2(\R)$ onto  $L^2_+(\R)$.

As emphasized in \cite{KLV}, the great virtue of the pair \eqref{Lax:BO} is the fact that \eqref{BO} is equivalent not only to the Lax equation $\partial_t L = [P,L]$, but also to
\begin{equation}\label{BO virtue}
\partial_t q_+(t) = P_{q(t)}^\textsc{bo} q_+(t).
\end{equation}
In particular, this perspective gives an immediate explanation of why
\begin{equation}\label{beta defn}
\beta^{\textsc{bo}}(\kappa;q): =\langle q_+, (L_q^\textsc{bo} + \kappa)^{-1} q_+ \rangle
\end{equation}
is conserved under the \eqref{BO} flow.  Notice that we have followed the model of \cite{KLV} in writing the spectral parameter as $-\kappa$ with the understanding that $\kappa\gg 1$; correspondingly, we write the free resolvent and the logarithm of the renormalized perturbation determinant as
\begin{equation}\label{BO stuff}
R_0^\textsc{bo}(\kappa):=(L_0^\textsc{bo} + \kappa)^{-1} \ \ \text{and}\ \  \alpha^\textsc{bo}(\kappa;q):= -\log \det_2\bigl(1- R_0^\textsc{bo}(\kappa)C_+q\bigr),
\end{equation}
respectively.

We have chosen to mirror the notation in \cite{KLV} because this paper provides a lengthy discussion of how \eqref{beta defn} is connected to the many different approaches that have been developed to address conservation laws for \eqref{BO}.  In particular, \cite[Lemma~4.13]{KLV} provides a form of Jost--Pais Theorem by showing how \eqref{beta defn} is linked to the perturbation determinant:
\begin{align}\label{BO JP1}
\beta^\textsc{bo}(\kappa;q) = - 2\pi \kappa \tfrac{\partial}{\partial \kappa} \alpha^\textsc{bo}(\kappa;q).
\end{align}

One of the goals of Section~\ref{S:BO} is to give a new proof of this fact as an $h\to\infty$ limit of Theorem~\ref{T:main}.  The more interesting part of this analysis concerns the transmission coefficient.  In fact, there is a conundrum to be resolved!  Consider the Lax operator \eqref{Lax:BO} with $q\equiv 0$.  All plane waves have wave numbers $\xi\geq 0$ and so propagate in the same direction; there is no possibility of reflection!  This compels us to ask what becomes of the transmission coefficient in this limit.  The answer is that $\mathfrak t(\zeta)$ converges to unity and the conservation laws hide in the next term in the asymptotic expansion:

\begin{theorem}\label{T:BO}  Fix $q\in C^\infty_c(\R)$ and $\kappa>0$ satisfying $-\kappa\notin\sigma(L_q^\textsc{bo})$. Let $\zeta_\pm(\kappa)\in \Omega_h^\pm$ be the unique solutions to $-\kappa = f_h(\zeta)$.  Then $-\kappa\notin\sigma(L_q)$ for $h$ sufficiently large and 
\begin{align}
	\label{E:BO i}
\tfrac{2\kappa h}{1-e^{-2h\zeta}}\int_\R q(x)\, dx &\,\longrightarrow\, \int_\R q(x)\,dx,\\
	\label{E:BO t}
\mp 2i\kappa h \log[ \mathfrak t(\zeta_\pm) ] &\,\longrightarrow\, \beta^\textsc{bo}(\kappa;q)+ \int_\R q(x)\, dx,\\
	\label{E:BO d}
\mp 2i\kappa h \log\biggl[ \frac{\det_2\bigl(1 - R_0(\zeta\mp\frac{i\pi}h) q\bigr)}{\det_2\bigl(1- R_0(\zeta) q\bigr)} \biggr]
	&\,\longrightarrow\, -2\pi \kappa \tfrac{d\ }{d\kappa} \alpha^\textsc{bo}(\kappa;q),
\end{align} 
as $h\to \infty$. In this way, Theorem~\ref{T:main} reproduces \eqref{BO JP1} in the limit $h\to\infty$.  See also \eqref{BO stuff} and \eqref{beta defn} for the definition of the functions $\alpha^\textsc{bo}(\kappa;q) $ and $\beta^\textsc{bo}(\kappa;q)$.
\end{theorem}

In the original proof of \eqref{BO JP1}, the parameter $\kappa$ is taken very large and positive.  This results in a determinant that is very close to unity and the logarithm is understood as the canonical branch: $\Log(1+\eps)=\eps + O(\eps^2)$.  As it is no impediment to the approach we take in this paper, we have allowed the spectral parameter to vary much more widely.  To short-cut any inconvenience with regard to branches of the logarithm, we simply define the logarithmic derivative thus:
$$
\tfrac{d\ }{d\kappa} \log F(\kappa) := \lim_{\eps\to 0} \tfrac1\eps \Log[ F(\kappa+\eps) / F(\kappa) ] = F'(\kappa)/F(\kappa).
$$

The proof of Theorem~\ref{T:BO} appears at the end of Section~\ref{S:BO} after two preliminary steps: Lemma~\ref{L:Op:BO}, which describes the large-$h$ behavior of the resolvent of the Lax operator, and Proposition~\ref{P:BO}, which describes the analogous behavior of $\phi_\pm$ and that of the transmission coefficient $\mathfrak t$.
The restriction $\kappa>0$ in Theorem~\ref{T:BO} is purely for expository reasons.  With additional labor, Lemma~\ref{L:Op:BO} can be adapted to complex $\kappa$, after which the remainder of the argument carries over without change.

In Section~\ref{S:KdV}, we study the shallow-water limit, which may be modeled with \eqref{KdV}.  Naturally, the wave height must be adjusted in unison with the total depth of the fluid; correspondingly, we adopt the new unknown
\begin{equation}\label{u(t) defn}
u(t,x) := \tfrac1h q\bigl(\tfrac{3t}h,x\bigr).
\end{equation}
This definition differs by numerical factors from other studies of the shallow-water limit, such as \cite{ABFS,MR4866612,HGKV:ILW,MR4752554,MR4878416}, due to our different normalization of \eqref{KdV}.  

Our goal in Section~\ref{S:KdV} is to show how both sides of \eqref{E:main} converge in this regime; see Theorem~\ref{T:KdV}. This allows us to verify that Theorem~\ref{T:main} does indeed reduce to the classical Theorem~\ref{T:JP} as the Lax operator $L_q$ converges to the Schr\"odinger operator \eqref{Lax:KdV}.

In closing this introduction, let us remark that there should be analogues of the Jost--Pais Theorem for many other integrable systems: one would like to understand the link between the conservation laws arising in inverse scattering and those coming from the perturbation determinant.  In addition to the models discussed here, such results are already known for Toda and for systems with Lax operators of ZS--AKNS type, such as (NLS), (mKdV), and (dNLS).  In these cases, the original approach of Jost--Pais carries over; see, for example, \cite[Lemma~3.3]{HGKV:Equi} and \cite[Lemma~2.8]{KillipFest}.

In fact, the perturbation determinant appears to be a universal approach to conservation laws for integrable systems admitting Lax pairs.  It has even proven applicable to models on the torus, where inverse scattering is not.  Nevertheless, one may still ask for analogues of the Jost--Pais Theorem in the periodic setting.  For example, in the \eqref{KdV} setting, the perturbation determinant can be linked to the Lyapunov exponent of eigenfunctions; see \cite[Prop.~6.4]{KV}.

\subsection{Notation}\label{SS:note}

We write $A\lesssim B$ to indicate $A\leq CB$ for an absolute constant $C>0$ whose specific value is not important.  Occasionally, we use subscripts to indicate dependence of the constant $C$ on other parameters; for instance, we write $A \lesssim_{\alpha, \beta} B$ when $A \leq CB$ with the constant $C>0$ depending on $\alpha, \beta$.  We write $A\simeq B$ when $A\lesssim B$ and $B\lesssim A$.

Our convention for the Fourier transform is
\begin{align*}
\widehat f(\xi) = \tfrac{1}{\sqrt{2\pi}} \int_\R e^{-i\xi x} f(x)\,dx  \qtq{so} f(x) = \tfrac{1}{\sqrt{2\pi}} \int_\R e^{i\xi x} \widehat f(\xi)\,d\xi.
\end{align*}
This is unitary under the inner product 
\[
    \langle f, g \rangle =\int_{\R} \overline{f(x)} g(x)\, dx
\]
and we have the  convolution identity $\widehat{fg} = (2\pi)^{-\frac12} \widehat f *\widehat g.$

An operator is said to be Hilbert--Schmidt if it is compact and its singular values are square summable.  It is not difficult to verify that this abstract formulation coincides with the historical one: an operator $A$ is Hilbert--Schmidt if it is repesented by an integral kernel $K(x,y) \in L^2_{x,y}$. The Hilbert--Schmidt norm is given by
\begin{equation}\label{HS norm}
\| A \|_{\mathfrak I_2}^2 = \iint | K(x,y)|^2 \,dx\,dy.
\end{equation}

Trace class operators are precisely those that can be written as the product of two Hilbert--Schmidt operators. Moreover, if $A=A_1 A_2$ with $A_1, A_2\in \mathfrak I_2$ having integrals kernels $K_1, K_2$, then we define the \emph{trace} of $A$ as
$$
\tr(A) = \iint \overline{K_1(x,y)} K_2(y,x) \, dx\, dy.
$$
Correspondingly, if $A\in \mathfrak I_2$, then 
$$
\| A \|_{\mathfrak I_2}^2 =  \tr(A^* A).
$$
For further discussion of the trace functional on Hilbert space and the associated ideals $\I_p$, we recommend the book \cite{MR2154153}.

A basic fact that we will use later is that
\begin{equation}\label{det AB}
\det_2\bigl(1+AB\bigr) = \det_2\bigl(1+BA\bigr),
\end{equation}
so long as $AB,BA \in \mathfrak I_2$. Moreover, if $A,B\in\mathfrak I_2$ with $1-A$ invertible, then
\begin{align}\label{gen id}
\det_2\bigl(1 - A + B \bigr)= \det_2\bigl(1  + \tfrac{1}{1-A} B \bigr) \det_2\bigl(1 - A\bigr) \exp\bigl\{ \tr\bigl( \tfrac{A}{1-A} B\bigr)\bigr\}.
\end{align}
The identity \eqref{gen id} is simply the product rule for renormalized determinants.  By choosing $B$ to be the rank-one operator
\begin{equation}\label{rank one}
|\phi \rangle \langle \psi| : f \mapsto \langle \psi, f \rangle \phi,
\end{equation}
expressed here in Dirac notation, we deduce that
\begin{align}\label{id}
\det_2\bigl(1 - A +|\phi \rangle \langle \psi|\bigr)= \bigl[1+ \langle \psi, (1-A)^{-1}\phi\rangle\bigr] e^{-\langle \psi , \phi\rangle}\det_2\bigl(1 - A\bigr),
\end{align}
whenever $A\in\mathfrak I_2$ with $1-A$ invertible and $\phi, \psi\in L^2(\R)$.

\subsection*{Acknowledgements}
Much of this work was completed while Y.~W. and R.~K. were visiting the Mathematics Department at UCLA; we are grateful for their hospitality.
R.~K. was supported by NSF grant DMS-2452346 and the project ANR-25-CFFS-0004 of the France 2030 program.  M.~V. was supported in part by NSF grant DMS-2348018.
Y.~W. was supported by NSF grant DMS-2407847.

\section{Perturbation determinants and scattering theory}\label{S:2}

The overall plan for this section is to prove the Hilbert--Schmidt bound needed to make sense of the perturbation determinant and then to turn to scattering theoretic considerations, such as the construction of the Jost-like solutions $\phi_\pm$.  However, there is one way in which scattering theory takes preeminence: throughout the paper, the dominant variable will be the wave number $\zeta\in\C$, rather than the spectral parameter $z\in\C$.  As noted already, these are related by \eqref{f defn}, that is, $z=f_h(\zeta)$.

The mappings $f_h(\zeta)$ have the simple scaling property
\begin{equation}\label{f scale}
f_h(\zeta) = \tfrac1{2h} f( 2h\zeta ) \qtq{with} f(\zeta) := \zeta + e^{-\zeta}-1,
\end{equation}
as well as the following translation and Schwarz reflection symmetries: 
\begin{equation}\label{f symm}
f_h\bigl(\zeta + \tfrac{in\pi}{h}\bigr) = f_h\bigl(\zeta\bigr) + \tfrac{in\pi}{h} \qtq{and}
f_h\bigl(\overline{\zeta + \tfrac{in\pi}{2h}}\bigr) = \overline{f_h\bigl(\zeta + \tfrac{in\pi}{2h}\bigr)  } \quad\text{for all $n\in\Z$}.
\end{equation}

For each $n\in\Z$, $f_h(\zeta)$ is a conformal mapping of the strip
\begin{equation}\label{f conf}
\tfrac{n\pi}h < \Im \zeta < \tfrac{(n+1)\pi}h \qtq{ onto} \C \setminus\bigl\{x+iy:x\geq 0,\  y=\tfrac{n\pi}h \text{ or } y=\tfrac{(n+1)\pi}h\bigr\},
\end{equation}
as discussed in \cite{Kober}.  Extending to the boundaries yields a `double covering' of each slit: both $[0,\infty)+\tfrac{in\pi}h$ and $(-\infty,0]+\tfrac{in\pi}h$ are mapped bijectively onto $[0,\infty)+\tfrac{in\pi}h$.  In this way, we see that for each $z\in\C$ there are infinitely many $\zeta\in\C$ so that $z=f_h(\zeta)$, as was mentioned in the introduction.

In what follows, we will be restricting $\zeta$ to lie in one of the two strips \eqref{om pm}.  To understand the rationale for this, we must first discuss the operators $L_q$.

The Lax operator $L_0$ associated to $q\equiv 0$ is the Fourier multiplier with symbol $\xi\mapsto f_h(\xi)$. It is self-adjoint, indeed positive semi-definite, with domain 
\begin{equation}\label{domain}
\mathcal D(L_0) = \{ u\in H^1(\R) : e^{-2h\xi}\mkern 2mu \widehat u(\xi) \in L^2(\R)\} .
\end{equation}
The spectrum $\sigma(L_0)=[0,\infty)$ is purely absolutely continuous and of multiplicity two: for each $E\in(0,\infty)$ there are two $\xi\in\R$ so that $f_h(\xi)=E$, one positive and the other negative.  In the dynamical framework $i\partial_t \phi = L_0\phi$, these represent the right- and left-moving waves, respectively.

A key notion in time-independent scattering theory is the principle of limiting absorption.  One moves the spectral parameter $z$ up or down a small amount in complex plane, taking it out of the spectrum.  Following each of the solutions identified in the previous paragraph we are lead to two \emph{principal} solutions $\zeta$ to the equation $z=f_h(\zeta)$, one in each strip $\Omega_h^\pm$.  If $\zeta\in\Omega_h^+$, then $\Im \zeta > 0$ and $e^{i\zeta x}$ represents a (weakly) attenuated wave propagating rightward, while $\zeta\in\Omega_h^-$ correspond to a wave traveling toward the left.  We will return to this point later when we construct $\phi_\pm$.  First, we need to know more about the Lax operator $L_q$ when $q\not\equiv 0$.  Our understanding of this operator is built on the following foundation:

\begin{lemma}\label{L:Rel HS}
Suppose $q\in L^2(\R)$, $\zeta\in \Omega_h^\pm$, and $z=f_h(\zeta)$. Then
\begin{align}\label{Rel HS}
\bigl\| R_0(\zeta) q \bigr\|_{\mathfrak I_2}^2 \lesssim \frac{\| q\|_{L^2}^2 }{\dist\bigl(z,[0,\infty)\bigr)}
	\Bigl[ 1 + h^{-\frac12} \dist\bigl(z,[0,\infty)\bigr)^{-\frac12} \Bigr].
\end{align}
Here $R_0(\zeta)$ denotes the free resolvent; see \eqref{R0 defn}.
\end{lemma}

\begin{proof}
Direct computation shows that
\begin{align}\label{fg}
\bigl\| R_0(\zeta) q \bigr\|_{\mathfrak I_2}^2 = \tfrac1{2\pi} \| q\|_{L^2}^2 \int \tfrac{d\xi}{| f_h(\xi)-z|^{2}} , 
\end{align}
so we must estimate this integral.  Changing variables according to $\eta=2h\xi$ and $2hz=x+iy$, we obtain
\begin{align*}
\int \tfrac{d\xi}{| f_h(\xi)-z|^{2}} = \int \tfrac{2h\, d\eta}{[ f(\eta)- x]^{2}+y^2}
	= 2h \int_0^\infty \bigl| \{ \eta: |f(\eta) - x|\leq \lambda \}\bigr| \,\tfrac{2\lambda d\lambda}{(\lambda^2+y^2)^2},
\end{align*}
where $f(\eta)=\eta + e^{-\eta} - 1$ as in \eqref{f scale}.

Elementary considerations show that 
$$
\bigl| \{ \eta: |f(\eta) - x|\leq \lambda \}\bigr| \lesssim \sqrt{\lambda} + \lambda,
$$
uniformly for $x\in\R$ and $\lambda\geq 0$.  This is sharp when $x=0$; otherwise, slight improvement is possible for $\lambda$ small.  We favor simplicity, noting only that the set is empty if $x<-\lambda$.  Writing $A:={\max\{0,-x\}}$, we are led to the estimate
\begin{align}\label{secret}
\int \tfrac{d\xi}{| f_h(\xi)-z|^{2}} \lesssim h \int_A^\infty  \tfrac{\lambda^{3/2} + \lambda^2}{(\lambda^2+y^2)^2} \, d\lambda 
	\lesssim h(A+|y|)^{-3/2} + h(A+|y|)^{-1}.
\end{align}
To deduce \eqref{Rel HS}, we simply note that $(A+|y|)$ is comparable to $2h\dist(z,[0,\infty))$.
\end{proof}

\begin{corollary} If $q\in L^2(\R)$, the operator $L_q$ is self-adjoint with domain \eqref{domain} and its essential spectrum is $[0,\infty)$.  Moreover,
\begin{align}\label{Lq below}
 \inf \sigma( L_q )  \gtrsim -  \Bigl[\|q\|_{L^2}^2  + h^{-\frac13} \|q\|_{L^2}^{\frac43} \Bigr] 
\end{align}
in the sense of quadratic forms.
\end{corollary}

\begin{proof}
Lemma~\ref{L:Rel HS} shows that $L_q$ is a relatively compact perturbation of $L_0$.  Thus the claims in the first sentence follow from the Kato--Rellich and Weyl Theorems.

Using \eqref{Rel HS}, we see that
$$
\| (L_0+\kappa)^{-1} q \|_\op < \tfrac12 \qtq{whenever} \kappa \gg \Bigl[\|q\|_{L^2}^2  + h^{-\frac13} \|q\|_{L^2}^{\frac43} \Bigr].
$$
It then follows that for the corresponding spectral parameters $z=-\kappa$, the resolvent $R_q$ of $L_q$ can be constructed via
$$
R_q = \bigl[ 1 - R_0 q ]^{-1} R_0.
$$
Thus, every such $z$ lies in the resolvent set of $L_q$ and \eqref{Lq below} follows.
\end{proof}

By showing that $R_0 q$ is a Hilbert--Schmidt operator, Lemma~\ref{L:Rel HS} guarantees the existence of the Hilbert renormalized determinants appearing in Theorem~\ref{T:main}.  One might imagine that this operator is trace class for $q\in C^\infty_c(\R)$; however, this is not the case (unless $q\equiv 0$).  This follows from \cite[Prop. 4.7]{MR2154153} and ultimately stems from the slow (linear) growth of $f_h(\xi)$ as $\xi\to+\infty$.

Let us turn now to stationary scattering theory. Our first task is the construction of the solutions \eqref{phi pm}:

\begin{prop}\label{P:phi}
Suppose $q\in C^\infty_c(\R)$, $\zeta\in \Omega_h^\pm$, and $z=f_h(\zeta)\notin \sigma(L_q)$. Then there is a unique solution $\phi_\pm=\phi_\pm(x; \zeta,q)$ to
\begin{equation}\label{phi pm'}
\phi_\pm = e^{i\zeta x} + R_0(\zeta) q \phi_\pm \qtq{with} \phi_\pm - e^{i\zeta x} \in L^2(\R). 
\end{equation}
\end{prop}

\begin{proof}
As $z\notin \sigma(L_q)$, the resolvent $R_q(\zeta)$ exists and is $L^2$-bounded.  Using the resolvent identity, we see that
\begin{equation}\label{404}
[ 1 - R_0(\zeta) q  ] [ 1 + R_q(\zeta)q ]  = 1 = [ 1 + R_q(\zeta)q ] [ 1 - R_0(\zeta) q  ] .
\end{equation}
Thus, the operator $1 - R_0(\zeta) q$ is invertible and
\begin{equation*}
[ 1 - R_0(\zeta) q  ]  \theta = R_0(\zeta) q e^{i\zeta x}
\end{equation*}
has a unique solution $\theta\in L^2(\R)$.  In this way, we see that
\begin{align}\label{405}
\phi_\pm := e^{i\zeta x}+ \theta = e^{i\zeta x} + [ 1 - R_0(\zeta) q  ]^{-1}R_0(\zeta) q e^{i\zeta x}= e^{i\zeta x} + R_q(\zeta) q e^{i\zeta x}
\end{align}
is the unique solution to \eqref{phi pm'}.
\end{proof}

The exact same argument can be applied with the \eqref{KdV} Lax operator \eqref{Lax:KdV}.  Taking advantage of the explicit form of the free Green's function in the \eqref{KdV} setting, we see that we have unique solutions
\begin{equation}\label{KdV phi}
\phi_\pm^\textsc{kdv}(x) = e^{i\zeta x} + R_0^\textsc{kdv}(\zeta) q \phi_\pm^\textsc{kdv} = e^{i\zeta x} \pm \tfrac{i}{2\zeta} \int e^{\pm i\zeta|x-y|} q(y) \phi_\pm^\textsc{kdv}(y)\,dy,
\end{equation}
whenever $\pm\Im \zeta >0$ and $\zeta^2\notin\sigma(L_q^\textsc{kdv})$. Here, $R_0^\textsc{kdv}(\zeta):= (D^2-\zeta^2)^{-1}$.

Notice how the free Green's function factors into a function of $x$ times a function of $y$ in each of the regions $x<y$ and $x>y$.  This is called semi-separability and was the key to the original Jost--Pais proof of Theorem~\ref{T:JP}.  Unfortunately, this property does not hold in the \eqref{ILW} setting. 

Using the semi-separability property in the \eqref{KdV} setting, we see that for $q\in C^\infty_c(\R)$ we have
\begin{equation}\label{KdV phi_+ +}
\phi_+^\textsc{kdv}(x) = e^{i\zeta x}\biggl[ 1  + \tfrac{i}{2\zeta} \int e^{-i\zeta y} q(y) \phi_+^\textsc{kdv}(y)\,dy\biggr]
\qtq{when} x > \max\supp(q)
\end{equation}
and
\begin{equation}\label{KdV phi_+ -}
\phi_+^\textsc{kdv}(x) = e^{i\zeta x} + e^{-i\zeta x} \biggl[ \tfrac{i}{2\zeta} \int e^{i\zeta y} q(y) \phi_+^\textsc{kdv}(y)\,dy\biggr]
\qtq{when} x < \min\supp(q).
\end{equation}
This (and analogous formulae for $\phi_-$) explain the physical significance of these solutions.    The first term in \eqref{KdV phi_+ -} shows a wave of unit amplitude entering from the right.  The second term in \eqref{KdV phi_+ -} is the reflected wave, while \eqref{KdV phi_+ +} is the transmitted wave.  Thus, the physical  transmission coefficient is given by
\begin{equation}\label{t KdV}
 \mathfrak t^\textsc{kdv} (\zeta) := \lim_{x\to\pm\infty} e^{-i\zeta x}\phi_\pm^\textsc{kdv}(x)
 	= 1  \pm \tfrac{i}{2\zeta} \int e^{-i\zeta y} q(y) \phi_\pm^\textsc{kdv}(y)\,dy
\end{equation}
in each of the regions $\pm \Im\zeta >0$.

The asymptotics \eqref{KdV phi_+ +} and \eqref{KdV phi_+ -} also allow us to compare $\phi_\pm$ with the Jost solutions \eqref{Jost:KdV}; we find that for $\Im \zeta>0$,
$$
\phi_+^\textsc{kdv}(x;\zeta) = \mathfrak t^\textsc{kdv}(\zeta) \psi_+^\textsc{kdv}(x;\zeta) \qtq{and}	
	\phi_-^\textsc{kdv}(x;-\zeta) =  \mathfrak t^\textsc{kdv}(-\zeta) \psi_-^\textsc{kdv}(x;\zeta).
$$
Computing the Wronskian of theses two solutions at $x > \max\supp(q)$ and recalling \eqref{Jost:Wron} shows that $  \mathfrak t^\textsc{kdv}(\zeta) a^\textsc{kdv}_\textup{scatt}(\zeta)\equiv 1$, as promised in the introduction.

Several ingredients contributed to the simplicity of this \eqref{KdV} story: locality of the Lax operator, the concept of a  Wronskian, the semi-separability of the Green's function, and the simple $\zeta\mapsto -\zeta$ symmetry of wave numbers representing the same spectral parameter $\zeta^2$.  All of these ingredients evaporate in the \eqref{ILW} context.

Our remaining goal for this section is to provide a definition of the transmission coefficient in the \eqref{ILW} setting.  This will be fulfilled in Proposition~\ref{P:limit}, which provides an analogue of \eqref{t KdV}.  This result relies on our isolation of the following key property of the free Green's function for \eqref{ILW} that acts as a replacement for the semi-separability property from the \eqref{KdV} setting:

\begin{lemma}\label{L:two R0}
For all $\zeta\in \Omega_h^\pm$, we have
\begin{align*}
R_0(\zeta\mp\tfrac{i\pi}h) = e^{\pm\frac{\pi}h x}\Bigl[ R_0(\zeta)
	\mp \tfrac{i}{1 - e^{-2h\zeta}} \bigr| e^{i x \zeta} \bigr\rangle\bigl\langle e^{i x \overline\zeta} \bigr| \Bigr] e^{\mp\frac{\pi}h x}.
\end{align*}
\end{lemma}

\begin{proof}
For $\zeta\in \Omega_h^\pm$, the operator $R_0(\zeta)$ is a Fourier multiplier operator whose convolution kernel is given by
\begin{align}\label{433}
G_0(x; \zeta) = \tfrac1{2\pi} \int_\R \tfrac{e^{ix\xi} \,d\xi}{f_h(\xi) - f_h(\zeta)}.
\end{align}
While the integrand above is not absolutely integrable, an integration by parts argument exploiting the phase oscillation shows that the integral on RHS\eqref{433} may be defined as an improper Riemann integral. 

Recall that the analytic function $f_h$ is injective on each strip $\Omega_h^\pm$.  Consequently, the integrand in \eqref{433} has a unique pole at $\xi=\zeta$.  This pole is simple: $f'_h(\zeta) = 1 - e^{-2h\zeta} \neq 0$.  Combining this with the symmetry properties \eqref{f symm}, we make the following residue computation:
\begin{align*}
G_0(x; \zeta \mp \tfrac{i\pi}h) 
& = \tfrac1{2\pi} \int_\R \tfrac{e^{ix\xi}}{f_h(\xi) - f_h(\zeta) \pm \frac{i\pi}h}\, d\xi
= \tfrac1{2\pi} \int_{\R\pm\frac{i\pi}h} \tfrac{e^{ix\xi \pm \pi x/h}}{f_h(\xi) - f_h(\zeta)}\, d\xi \\
&=  e^{\pm \frac\pi{h} x} \Bigl[G_0(x; \zeta) \mp \tfrac{ie^{ix\zeta}}{f_h'(\zeta)}\Bigr].
\end{align*}
As $G_0$ represents the \emph{convolution kernel} of $R_0$, the claim follows.
\end{proof}

\begin{prop}\label{P:limit} For $q\in C^\infty_c(\R)$ and $\zeta\in \Omega_h^\pm$ with $f_h(\zeta)\notin \sigma(L_q)$, the limit
\begin{align}\label{b}
\mathfrak t(\zeta) :=\lim_{x\to \pm\infty} e^{-i\zeta x} \phi_\pm(x; \zeta,q) 
\end{align}
exists and satisfies
\begin{align}\label{a}
\mathfrak t(\zeta)  = 1 \pm \tfrac{i}{1-e^{-2h\zeta}}\int_\R e^{-iy\zeta}q(y)\phi_\pm(y; \zeta, q)\, dy.
\end{align}
\end{prop}

\begin{proof}
Combining \eqref{phi pm'} and Lemma~\ref{L:two R0}, we have
\begin{align*}
e^{-i\zeta x} \phi_\pm(x) ={}& 1 + e^{-i\zeta x} \bigl[R_0(\zeta)q \phi_\pm\bigr](x) \\
={}& 1 + e^{\mp(\frac{\pi}h\pm i\zeta) x} \bigl[R_0(\zeta\mp\tfrac{i\pi}h) \gamma\bigr](x)
	\pm \tfrac{i}{1 - e^{-2h\zeta}} \int_\R e^{-iy\zeta}q(y)\phi_\pm(y; \zeta, q)\, dy,  
\end{align*}
where $\gamma(x):=e^{\pm\frac{\pi}h x} q(x) \phi_\pm(x) \in L^2(\R)$ since $q\in C^\infty_c(\R)$.  

Recalling \eqref{domain}, we infer that
\begin{equation}
R_0(\zeta\mp\tfrac{i\pi}h) \gamma\in \mathcal D(L_0) \subseteq H^1(\R) \subseteq L^\infty(\R) .
\end{equation}
Our claims then follow because $|\Im \zeta| < \frac{\pi}h$ and so $\Re (\frac{\pi}h\pm i\zeta)>0$.
\end{proof} 

\section{Proof of Theorem~\ref{T:main}}\label{S:3}

The goal of this section is to prove Theorem~\ref{T:main}.  For the convenience of the reader, we recall the statement:

\begin{theorem}\label{T:JP'} Fix $q\in C^\infty_c(\R)$ and $\zeta\in \Omega_h^\pm$ with $f_h(\zeta)\notin\sigma(L_q)$.  Then the transmission coefficient \eqref{a} satisfies
\begin{align}\label{E:JP'}
\mathfrak t(\zeta;q) = 
\frac{\det_2\bigl(1 - R_0(\zeta\mp\frac{i\pi}h) q\bigr)}{\det_2\bigl(1- R_0(\zeta) q\bigr)} \exp\biggl\{\pm \tfrac{i}{1-e^{-2h\zeta}}\int_\R q(x)\, dx \biggr\}.
\end{align}
\end{theorem}

\begin{proof}
Notice that $\zeta\mp\frac{i\pi}h\in\Omega_h^\mp$ and so Lemma~\ref{L:Rel HS} guarantees that both renormalized determinants are well defined.  Moreover, as $f_h(\zeta)\notin\sigma(L_q)$, we know from \eqref{404} that $1-R_0(\zeta)q$ is invertible and so
\begin{align}\label{nz}
\det_2\bigl(1- R_0(\zeta) q\bigr)\neq 0.
\end{align}

Let $\chi\in C_c^\infty(\R)$ be a cutoff function such that $\chi(x) =1$ for all $x\in \supp (q)$ and let $E_\pm\in C_c^\infty(\R)$ be such that $E_\pm(x) = e^{\pm\frac{\pi}h x}$ for all $x\in \supp(\chi)$.

Writing $q= \chi q\chi$ and using \eqref{det AB} together with Lemma~\ref{L:Rel HS}, we get 
\begin{align*}
\det_2\bigl(1 - R_0(\zeta\mp\tfrac{i\pi}h) q\bigr)
&=\det_2\bigl(1 - R_0(\zeta\mp\tfrac{i\pi}h) \chi q\chi\bigr)
= \det_2\bigl(1 - \chi R_0(\zeta\mp\tfrac{i\pi}h) \chi q\bigr).
\end{align*}
By Lemma~\ref{L:two R0}, we have
\begin{align*}
\chi R_0(\zeta \mp\tfrac{i\pi}h) \chi = E_\pm\chi R_0(\zeta) E_\mp\chi  \mp \tfrac{i}{1-e^{-2h\zeta}} \bigl|E_\pm\chi e^{ix\zeta}\rangle \langle  e^{ix\overline \zeta}E_\mp\chi\bigr|
\end{align*}
and so, using again \eqref{det AB}, we deduce that
\begin{align*}
\det_2\bigl(1- R_0(\zeta\mp\tfrac{i\pi}h) q\bigr)
&=\det_2\Bigl(1 - E_\pm\chi R_0(\zeta) E_\mp\chi q \mp \tfrac{i}{1-e^{-2h\zeta}} \bigl|E_\pm\chi e^{ix\zeta}\rangle \langle  e^{ix\overline \zeta}E_\mp\chi q\bigr|\Bigr)\\
&=\det_2\Bigl(1 - \chi R_0(\zeta) \chi q  \pm \tfrac{i}{1-e^{-2h\zeta}} \bigl|\chi e^{ix\zeta}\rangle \langle e^{ix\overline \zeta}\chi q\bigr|\Bigr).
\end{align*}

To continue, we would like employ \eqref{id} with
$$
A= \chi R_0(\zeta) \chi q,\quad \phi= \tfrac{i}{1-e^{-2h\zeta}}\chi e^{ix \zeta}, \qtq{and} \psi=\pm e^{ix\overline \zeta}\chi q .
$$
This requires that $A\in\mathfrak I_2$ with $1-A$ invertible and $\phi, \psi\in L^2(\R)$.  That $A\in\mathfrak I_2$ follows from Lemma~\ref{L:Rel HS}. The  invertibility of $1-A$ follows from an application of \eqref{det AB} and \eqref{nz}:
$$
\det_2\bigl(1- \chi R_0(\zeta) q\chi\bigr) = \det_2\bigl(1- R_0(\zeta) q\bigr)\neq 0.
$$
That $\phi, \psi\in L^2(\R)$ follows from $\chi, q\in C^\infty_c(\R)$.

Proceeding as announced, we find
\begin{align*}
\det_2\bigl(1- R_0(\zeta\mp\tfrac{i\pi}h) q\bigr)
&=\Bigl[ 1 \pm \tfrac{i}{1-e^{-2h\zeta}} \bigl\langle e^{ix\overline \zeta}\chi q, \bigl[1- \chi R_0(\zeta) \chi q\bigr]^{-1} \chi e^{ix\zeta} \bigr\rangle\Bigr]\\
&\qquad\quad \times \det_2\Bigl(1 - R_0(\zeta) q\Bigr)  \exp\Bigl\{\mp\tfrac{i}{1-e^{-2h\zeta}}\int_\R q(x)\, dx \Bigr\}.
\end{align*}

Comparing this with \eqref{E:JP'} and \eqref{a}, we see that it remains to show that
\begin{align}\label{110}
 \bigl\langle e^{ix\overline \zeta}\chi q, \bigl[1- \chi R_0(\zeta) \chi q\bigr]^{-1} \chi e^{ix\zeta} \bigr\rangle = \int_\R e^{-iy\zeta}q(y) \phi_\pm (y; \zeta, q)\, dy.
\end{align}
Starting with \eqref{phi pm'}, we have
\begin{equation}
\chi \phi_\pm = \chi e^{i\zeta x} + \chi R_0(\zeta) q \chi^2 \phi_\pm  \qtq{and so} \bigl[1- \chi R_0(\zeta) \chi q\bigr] \chi \phi_\pm = \chi e^{i\zeta x}.
\end{equation}
This readily yields \eqref{110}, thereby completing the proof of the theorem.
\end{proof}

\section{The Benjamin--Ono limit}\label{S:BO}

The goal of this section is to understand how both sides of \eqref{E:main} behave in the deep-water limit, culminating in a proof of Theorem~\ref{T:BO}.  We begin with a discussion of the Benjamin--Ono Lax operator \eqref{Lax:BO}.

For $q\in C^\infty_c(\R)$, the operator $L_q^{\mathrm{BO}}$ is self-adjoint with domain
$$
\mathcal{D}(L_q^{\mathrm{BO}}) = H^1_+(\mathbb{R}):= H^1(\R)\cap L^2_+(\R).
$$
The essential spectrum is $[0, \infty)$ and the operator is bounded from below; indeed,
\begin{equation}\label{L below}
L_q^{\mathrm{BO}} \geq -\tfrac{1}{2\pi} \|q\|_{L^2}^2 
\end{equation}
in the sense of quadratic forms.  For all these assertions and much more, see \cite{BKV:OS,KLV,Sun2021,MR3484397}.  One reason we highlight \eqref{L below} here is that it gives a simple condition that guarantees the requirement $-\kappa\notin\sigma(L_q^\textsc{bo})$ from Theorem~\ref{T:BO}.

Supposing that $\kappa>0$ and $-\kappa\notin\sigma(L_q^\textsc{bo})$, we now consider the $h\to\infty$ behavior of $\zeta_\pm(\kappa)$, which are defined as the unique solutions to
\begin{align}\label{zeta eq"}
f_h(\zeta_\pm) = - \kappa \qtq{with} \zeta_\pm \in \Omega_h^\pm.
\end{align}

First, let us recast this equation as 
$$
e^{\Delta} - 1 + \tfrac1{2h\kappa}\Delta = \tfrac1{2h\kappa}[\pm i\pi -1 - \log(2h\kappa)] \qtq{where} \zeta_\pm= -\tfrac1{2h}\log(2h\kappa) \pm\tfrac{i\pi}{2h} - \tfrac1{2h}\Delta.
$$
The equation for $\Delta$ can be solved via the inverse function theorem in a neighborhood of $\Delta=0$, as long as the right-hand side is small.  In particular, we find that
\begin{align}\label{zeta asymp}
\zeta_\pm  = -\tfrac1{2h}\log(2h\kappa) \pm\tfrac{i\pi}{2h} + O\bigl(\tfrac{\log(2h\kappa)}{\kappa h^2}\bigr) \qtq{uniformly for} 2h\kappa\gtrsim 1
\end{align}
and thence also that 
\begin{align}\label{rat asymp}
\tfrac{1}{1-e^{-2h\zeta}}  = \tfrac{1}{2h\kappa} + O\bigl(\tfrac{\log(2h\kappa)}{(h\kappa)^2}\bigr) \qtq{uniformly for} 2h\kappa\gtrsim 1.
\end{align}

Our interest in \eqref{rat asymp} stems not only from the appearance of this quantity in \eqref{E:JP'}, but also in connection with \eqref{a}.  Indeed, the large-$h$ behaviour of the transmission coefficient is interesting because, as noted in the introduction, the Benjamin--Ono Lax operator only permits uni-directional plane waves. Reflection should therefore disappear in the $h\to \infty$ limit!

To understand this limit, we must first understand the large-$h$ behaviour of the solutions \eqref{phi pm'}.  This is the role of \eqref{1212} in the next lemma, while \eqref{1214} anticipates our needs in analyzing the ratio of perturbation determinants appearing in \eqref{E:JP'}.

\begin{lemma}\label{L:Op:BO} For $q\in C^\infty_c(\R)$ and $\kappa>0$, we have
\begin{align}\label{1212}
 R_0(\zeta_\pm)q \to  R_0^\textsc{bo}(\kappa)C_+ q \quad\text{in $\mathfrak I_2$ sense as $h\to\infty$;}
\end{align}
see \eqref{zeta eq"} and \eqref{BO stuff} for notation.  Furthermore,
\begin{align}\label{1214}
h \bigl[ R_0\bigl(\zeta_\pm\bigr) - R_0\bigl(\zeta_\pm\mp\tfrac{i\pi}{h}\bigr) \bigr] q  \to \pm i\pi \bigl[R_0^\textsc{bo}(\kappa)\bigr]^2 q
\quad\text{in $\mathfrak I_2$ sense as $h\to\infty$.}
\end{align}
\end{lemma}

\begin{proof}
The same calculation that yielded \eqref{fg} shows
\begin{align}\label{357}
\bigl\| R_0^{\mathrm{BO}}(\kappa)C_+q\bigr\|_{\mathfrak I_2}^2 = 
	\tfrac1{2\pi} \|q\|_{L^2}^2 \int_0^\infty \!\!\tfrac{d\xi}{|\xi+\kappa|^2} = \tfrac1{2\pi\kappa} \|q\|_{L^2}^2,
\end{align}
as well as
\begin{align}\label{358}
\bigl\| R_0(\zeta_\pm)q - R_0^{\mathrm{BO}}(\kappa)C_+q\bigr\|_{\mathfrak I_2}^2
	= \tfrac1{2\pi}\|q\|_{L^2}^2 \Bigl\| \tfrac{1}{f_h(\xi) + \kappa} - \tfrac{1_{[0,\infty)}(\xi)}{\xi+\kappa}\Bigr\|_{L^2_\xi(\R)}^2.
\end{align}

To proceed from \eqref{358}, we consider first the region $\xi\geq 0$:
\begin{align*}
\int_0^\infty \Bigl| \tfrac{1}{f_h(\xi) + \kappa} - \tfrac{1_{[0,\infty)}(\xi)}{\xi+\kappa}\Bigr|^2 \, d\xi
	&= \int_0^\infty \tfrac{|e^{-2h\xi}-1|^2}{4h^2 |f_h(\xi) + \kappa|^2 \, |\xi+\kappa|^2}\,d\xi\\
&\leq\int_0^\infty \tfrac{d\xi}{ h^2 \kappa^2 |\xi+\kappa|^2 } \lesssim h^{-2}\kappa^{-3}  . 
\end{align*}
In the region $\xi<0$, the Taylor expansion of $f_h(\xi)$ has only positive terms, from which it follows that $f_h(\xi)\geq h\xi^2$.  Thus,
\begin{align*}
\int_{-\infty}^0 \Bigl| \tfrac{1}{f_h(\xi) + \kappa} - \tfrac{1_{[0,\infty)}(\xi)}{\xi+\kappa}\Bigr|^2 \, d\xi
	&\leq \int_{-\infty}^0 \tfrac{d\xi}{[h\xi^2 + \kappa]^2} \lesssim h^{-\frac12}\kappa^{-\frac32}.
\end{align*}
Combining the estimates for these two regions in \eqref{358} proves \eqref{1212}.

We turn now to \eqref{1214}, noting that by \eqref{f symm}, $f_h( \zeta_\pm\mp\tfrac{i\pi}{h} ) = -\kappa \mp\tfrac{i\pi}{h}$.  Proceeding as above, we find
\begin{align}\label{359}
\bigl\|h \bigl[ R_0\bigl(\zeta_\pm\bigr)& - R_0\bigl(\zeta_\pm\mp\tfrac{i\pi}{h}\bigr) \bigr] q  \mp i\pi \bigl[R_0^\textsc{bo}(\kappa)\bigr]^2 q\bigr\|_{\mathfrak I_2}^2\notag\\
&= \tfrac1{2\pi}\|q\|_{L^2}^2 \Bigl\| \tfrac{h}{f_h(\xi) + \kappa} - \tfrac{h}{f_h(\xi) + \kappa \pm i\pi/h} \mp \tfrac{i\pi 1_{[0,\infty)}(\xi)}{(\xi+\kappa)^2}\Bigr\|_{L^2_\xi(\R)}^2.
\end{align}
For $\xi\geq 0$, bringing everything to a common denominator we may bound
\begin{align*}
\Bigl| \tfrac{h}{f_h(\xi) + \kappa} - \tfrac{h}{f_h(\xi) + \kappa \pm i\pi/h} \mp \tfrac{i\pi 1_{[0,\infty)}(\xi)}{(\xi+\kappa)^2}\Bigr|
	&\lesssim  \tfrac{h(\xi+\kappa)+1}{h^2\kappa^2(\xi+\kappa)^2}
\end{align*}
and so deduce that
\begin{align}\label{1}
\int_0^\infty \Bigl| \tfrac{h}{f_h(\xi) + \kappa} - \tfrac{h}{f_h(\xi) + \kappa \pm i\pi/h} \mp \tfrac{i\pi 1_{[0,\infty)}(\xi)}{h(\xi+\kappa)^2}\Bigr|^2 \,d\xi
	&\lesssim  h^{-2}\kappa^{-5} + h^{-4}\kappa^{-7} .
\end{align}
Using that $f_h(\xi)\geq h\xi^2$ when $\xi<0$, we may bound
\begin{align}\label{2}
\int_{-\infty}^0  \Bigl| \tfrac{h}{f_h(\xi) + \kappa} - \tfrac{h}{f_h(\xi) + \kappa \pm i\pi/h}\mp \tfrac{i\pi 1_{[0,\infty)}(\xi)}{(\xi+\kappa)^2}\Bigr|^2 \, d\xi
	&\leq \int_{-\infty}^0 \tfrac{\pi^2\,d\xi}{[h\xi^2 + \kappa]^4} \lesssim h^{-\frac12} \kappa^{-\frac72}. 
\end{align}
Inserting \eqref{1} and \eqref{2} into \eqref{359} yields \eqref{1214}.
\end{proof}

Using this lemma, we may now describe the leading order behaviour of both the solutions $\phi_\pm$ and the transmission coefficient as $h\to\infty$:

\begin{prop}\label{P:BO}
Fix $q\in C^\infty_c(\R)$ and $\kappa>0$ satisfying $-\kappa\notin\sigma(L_q^\textsc{bo})$.  Then $-\kappa\notin\sigma(L_q)$ for $h$ sufficiently large. Moreover, as $h\to\infty$,
\begin{align}\label{703}
\phi_\pm(x;\zeta_\pm,q) - e^{ix\zeta_\pm } &\to \bigl[R_q^{\textsc{bo}}(\kappa) q_+\bigr](x) \quad\text{in $L^2_x(\R)$ sense and}\\
\label{704}
h \bigl[ \mathfrak t(\zeta_\pm) - 1  \bigr] &\to\pm \tfrac{i}{2\kappa} \biggl[ \int_\R q(x)\, dx + \bigl\langle q_+ , R_q^{\textsc{bo}}(\kappa) q_+\bigr\rangle\biggr],
\end{align}
where $R_q^{\textsc{bo}}(\kappa):=( L_q^\textsc{bo} + \kappa)^{-1}$.
\end{prop}

\begin{proof}
As $-\kappa\notin\sigma(L_q^\textsc{bo})$ and $q\in  C^\infty_c(\R)\subset L^\infty(\R)$, we know that $R_q^\textsc{bo}(\kappa)C_+q$ defines a bounded operator on $L^2(\R)$. Thus, mimicking \eqref{404}, we see that $1 - R_0^\textsc{bo}(\kappa)C_+q$ is invertible.  Using \eqref{1212}, we then deduce that $1- R_0(\zeta_\pm) q$ is invertible for $h$ sufficiently large, which in turn implies $-\kappa\notin\sigma(L_q)$. Using \eqref{1212} again shows
\begin{align}\label{703a}
\bigl(1- R_0(\zeta_\pm) q\bigr)^{-1} \to \bigl(1- R_0^\textsc{bo}(\kappa)C_+ q\bigr)^{-1}
	\quad\text{in operator norm as $h\to\infty$.}
\end{align}

Next, let us choose $\chi\in C^\infty_c(\R)$ such that $\chi\equiv 1$ on $\supp q$.  From \eqref{zeta asymp}, we see that $\zeta_\pm(\kappa)\to0$ as $h\to\infty$ and consequently,
\begin{equation}\label{703b}
 e^{i x \zeta_\pm}\chi(x) \to \chi(x) \qtq{and} e^{i x \overline{\zeta_\pm}}q(x) \to q(x) \quad\text{in $L^2_x(\R)$ as $h\to \infty$.}
\end{equation}

Let us now turn our attention more directly toward \eqref{703}.  Recalling \eqref{405}, we have
$$
\phi_\pm(x;\zeta_\pm,q) - e^{ix\zeta_\pm}  = [ 1 - R_0(\zeta_\pm) q  ]^{-1}  R_0(\zeta_\pm) q e^{ix\zeta_\pm } \chi.
$$
We now apply \eqref{703a}, \eqref{1212}, and \eqref{703b} to the corresponding factors to deduce
$$
\phi_\pm(x;\zeta_\pm,q) - e^{ix\zeta_\pm} \to  \bigl(1- R_0^\textsc{bo}(\kappa)C_+ q\bigr)^{-1} R_0^\textsc{bo}(\kappa)C_+ q\chi \quad\text{in $L^2_x(\R)$ as $h\to \infty$.}
$$
The claim \eqref{703} then follows from $q\chi=q$ and the resolvent identity.

To begin our proof of \eqref{704}, we rewrite \eqref{a} as
\begin{align*}
h \bigl[ \mathfrak t(\zeta_\pm) - 1  \bigr]  =  \pm \tfrac{ih}{1-e^{-2h\zeta_\pm}} \biggl[ \, \int_\R q(x)\, dx + \int_\R e^{-iy\zeta_\pm}q(y) [\phi_\pm(y; \zeta_\pm, q)-e^{iy\zeta_\pm}]\, dy \biggr].
\end{align*}
The large-$h$ behaviour of the prefactor is described by \eqref{rat asymp}, while $\int q$ is independent of $h$.  To understand the last term in the brackets, we write it first as an inner product and then apply \eqref{703b} and \eqref{703} to obtain
\begin{align*}
\bigl\langle e^{i x \overline{\zeta_\pm}}q(x), \, \phi_\pm(x; \zeta_\pm, q)-e^{ix\zeta_\pm}\bigr\rangle \to 
	\bigl\langle q,  R_q^{\textsc{bo}}(\kappa) q_+ \bigr\rangle \quad \text{as $h\to \infty$.}
\end{align*}
Note that $R_q^{\textsc{bo}}(\kappa) q_+$ lies in the Hardy space $L^2_+(\R)$ and so is orthogonal to $q-q_+$.  Thus, \eqref{704} follows.
\end{proof}

We are now ready to complete the 

\begin{proof}[Proof of Theorem~\ref{T:BO}]
The claim \eqref{E:BO i} follows immediately from \eqref{rat asymp}. From \eqref{704} and the fact that $\log(1+\eps)=\eps +O(\eps^2)$ as $\eps\to 0$, we find
\begin{align*}
\mp 2i \kappa h \log\bigl[ \mathfrak t(\zeta_\pm;q)\bigr] \to \int_\R q(x)\, dx + \bigl\langle q_+,  R_q^{\textsc{bo}}(\kappa) q_+ \bigr\rangle = \int_\R q(x)\, dx + \beta^\textsc{bo}(\kappa;q)
\end{align*}
as $h\to \infty$. This proves \eqref{E:BO t}, leaving us only to verify \eqref{E:BO d}:
\begin{align}\label{777}
\mp 2i\kappa h \log\biggl[ \frac{\det_2\bigl(1 - R_0(\zeta_\pm \mp\frac{i\pi}h) q\bigr)}{\det_2\bigl(1- R_0(\zeta_\pm) q\bigr)} \biggr]
	\to -2\pi \kappa \tfrac{d\ }{d\kappa}\mkern2mu \alpha^\textsc{bo}(\kappa;q)\quad\text{as $h\to \infty$}.
\end{align}
Recall from \eqref{BO stuff} that $\alpha^\textsc{bo}(\kappa;q)$ is defined as the logarithm of the renormalized Benjamin--Ono perturbation determinant.

First, let us employ \eqref{gen id} to make a general statement about the logarithmic differentiability of the renormalized determinant:  If $A,B \in \mathfrak I_2$ and $1-A$ is an invertible operator, then
\begin{align*}
\log\biggl[  \frac{ \det_2\bigl(1 - A + B \bigr) }{ \det_2\bigl(1 - A\bigr) } \biggr]
	= \tr\bigl( \tfrac{A}{1-A} B\bigr) + O\Bigl( \bigl\| \tfrac{1}{1-A} B\bigr\|_{\mathfrak I_2}^2 \Bigr)
		\qtq{for} \bigl\| \tfrac{1}{1-A} B\bigr\|_{\mathfrak I_2}\lesssim 1.
\end{align*}

Now we apply this with $A=R_0^\textsc{bo}(\kappa)C_+ q$ and $-\kappa\notin \sigma(L_q^\textsc{bo})$, so that $1-A$ is invertible; that $A\in \mathfrak I_2$ follows from \eqref{357}.  For $B$, we choose the one parameter family of operators
$$
B(\eps)=R_0^\textsc{bo}(\kappa)C_+ q - R_0^\textsc{bo}(\kappa+\eps)C_+ q \in\mathfrak I_2,
$$
noting that
$$
\eps^{-1} B(\eps) \to \bigl[R_0^\textsc{bo}(\kappa)\bigr]^2C_+ q \quad\text{in $\mathfrak I_2$ sense as $\eps\to0$.}
$$
In this way, we find that  
\begin{align*}
\lim_{\eps\to 0} \tfrac{1}\eps \log\biggl[  \frac{ \det_2\bigl(1 - R_0^\textsc{bo}(\kappa+\eps)C_+ q\bigr) }{ \det_2\bigl(1 - R_0^\textsc{bo}(\kappa)C_+ q\bigr) } \biggr]
	= \tr\Bigl( \tfrac{ R_0^\textsc{bo}(\kappa)C_+ q }{ 1 - R_0^\textsc{bo}(\kappa)C_+ q} \bigl[R_0^\textsc{bo}(\kappa)\bigr]^2 C_+ q \Bigr).
\end{align*}
Comparing with \eqref{BO stuff}, we recognize the left-hand side here as the $\kappa$-derivative of $-\alpha^\textsc{bo}(\kappa;q)$.  The right-hand side can be considerably simplified by cycling the trace and using resolvent identities.  In this way, we find
\begin{align}\label{801}
- \tfrac{d\ }{d\kappa} \alpha^\textsc{bo}(\kappa;q) = \tr\Bigl( R_0^\textsc{bo}(\kappa)C_+ q R_q^\textsc{bo}(\kappa) C_+ q R_0^\textsc{bo}(\kappa) \Bigr).
\end{align}

Let us now apply the same argument using
\begin{align}\label{807}
A = R_0\bigl(\zeta_\pm\bigr) q \qtq{and} B = R_0\bigl(\zeta_\pm\bigr)q - R_0\bigl(\zeta_\pm \mp\tfrac{i\pi}h\bigr)q,
\end{align}
as well as  Lemma~\ref{L:Op:BO} to understand the $h\to\infty$ behavior of these operators.  In this way, we find
\begin{align}\label{877}
h \log\biggl[ \frac{\det_2\bigl(1 - R_0(\zeta_\pm \mp\frac{i\pi}h) q\bigr)}{\det_2\bigl(1- R_0(\zeta_\pm) q\bigr)} \biggr]
	\to \pm i\pi \tr\Bigl( R_0^\textsc{bo}(\kappa)C_+ q R_q^\textsc{bo}(\kappa) C_+ q R_0^\textsc{bo}(\kappa) \Bigr)
\end{align}
as $h\to \infty$.

Combining \eqref{801} and \eqref{877}, we finally arrive at \eqref{777}.
\end{proof}

\section{The Korteweg--de Vries limit}\label{S:KdV}

For a physically reasonable shallow-water limit, one must maintain the wave amplitude proportional to the fluid depth.  Correspondingly, in this section we will adopt the new unknown 
\begin{equation}\label{u defn}
u(x) := \tfrac1h q(x).  
\end{equation}
We also adopt a commensurate modification of the Lax operator and its resolvent:
\begin{equation}\label{breve}
\breve L_u := \tfrac{1}{h} L_q \qtq{and correspondingly,} \breve R_u(\zeta) := \bigl( \breve L_u - \tfrac{1}{h}f_h(\zeta) \bigr)^{-1} = h R_q(\zeta). 
\end{equation}
We choose the breve accent because it means `short'. From \eqref{u defn} and \eqref{breve}, we have
\begin{equation}\label{yay}
R_0(\zeta) q = h R_0(\zeta) u = \breve R_0(\zeta) u   .
\end{equation}

The remainder of this section has been arranged to loosely parallel Section~\ref{S:BO}. We first construct $\zeta_\pm$ and prove operator convergence in Lemma~\ref{L:KdV}. We then prove convergence of $\phi_\pm$ in Proposition~\ref{P:KdV}.  Finally, Theorem~\ref{T:KdV} presents the $h\to 0$ limits of both  sides of the identity \eqref{E:main}.

\begin{lemma}\label{L:KdV}
Fix $q\in C^\infty_c(\R)$ and $k\in \C$ with  $\pm\Im k >0$.  For $h$ sufficiently small, depending only on $k$, there is a unique $\zeta_\pm \in \Omega_h^\pm$ with
\begin{equation}\label{zeta of k}
\tfrac1h f_h(\zeta_\pm) = k^2.
\end{equation}
Moreover,
\begin{equation}\label{zeta of k'}
\zeta_\pm = k \bigl[ 1 + O\bigl(2h|k|\bigr)\bigr]  	\qtq{as} h \to 0.
\end{equation}
Writing $R_0^\textsc{kdv}(k) := (D^2 - k^2)^{-1}$, we also have
\begin{align}\label{872}
R_0(\zeta_\pm)q = \breve R_0(\zeta_\pm) u  \to R_0^\textsc{kdv}(k)u \quad\text{in $\mathfrak I_2$ sense as $h\to 0$.}
\end{align}
\end{lemma}

\begin{proof}
To address \eqref{zeta of k} and \eqref{zeta of k'}, we define
\begin{equation}\label{g defn}
g(x) = \sum_{\ell\geq3} \tfrac{2}{\ell!} (-x)^{\ell-2}  \qtq{so that}  \tfrac1h f_h(\zeta) = \zeta^2[ 1 + g(2h\zeta) ].
\end{equation}
Applying the inverse function theorem in a neighborhood of the origin, we see that
\begin{equation}
2h\zeta \sqrt{1 + g(2h\zeta) } = 2hk \qtq{has a solution} \zeta = k \bigl[ 1 + O\bigl( 2h|k|\bigr)\bigr],
\end{equation}
provided $2h|k| \ll 1$.  It is now elementary to verify that this provides the $\zeta_\pm$ we seek. Uniqueness follows from the injectivity of $f_h$ on $\Omega_h^\pm$.

Direct computation shows
\begin{align}\label{872'}
\bigl\| \breve R_0(\zeta_\pm) u  - R_0^\textsc{kdv}(k)u\bigr\|_{\mathfrak I_2}^2
	=\tfrac1{2\pi}\|u\|_{L^2}^2 \int_{-\infty}^\infty \frac{\xi^4 g(2h\xi)^2}{|f_h(\xi)/h-k^2|^2|\xi^2 - k^2|^2}\,d\xi.
\end{align}
In contrast to the proof of Lemma~\ref{L:Op:BO}, where we gave very quantitative estimates for the corresponding integrals under the restriction $\kappa>0$, we now argue more qualitatively: As $k^2\notin [0,\infty)$, which is the range of both $\xi^2$ and $f_h(\xi)/h$, we have
\begin{equation}\label{eugh}
|\xi^2-k^2|^2 \simeq_k (\xi^2 + 1)^2 \qtq{and} |f_h(\xi)/h-k^2|^2 \simeq_k h^{-2} [f_h(\xi)+ h]^2,
\end{equation}
uniformly for $\xi\in\R$.

Let us first consider the integral in \eqref{872'} in the region $\xi>0$.  By means of the representation
\begin{equation}\label{g x>0}
1 + g(x) = \tfrac2{x^2}\int_0^x \int_0^t e^{-s}\,ds\,dt, \qtq{we see that} -1\leq g(x)\leq 0 \quad\text{for $x>0$.}
\end{equation}
This guarantees that $\xi^2 |g(2h\xi)| \leq \xi^2 + 1$. Using \eqref{eugh} and  \eqref{secret} (with $z=-h$ and so $A=h$ and $y=0$), we get
\begin{align}\label{872>0}
\int_0^\infty \frac{\xi^4 g(2h\xi)^2\,d\xi}{|f_h(\xi)/h-k^2|^2|\xi^2 - k^2|^2}
	\lesssim_k \int_0^\infty \frac{h^2\,d\xi}{|f_h(\xi) + h|^2} \lesssim_k h^{\frac32} + h^2.
\end{align}

In the region $\xi\leq 0$, the series \eqref{g defn} yields
\begin{align}\label{g x<0 1}
	0 \leq  g(2h\xi) \leq 2h|\xi| \cdot [1+g(2h\xi)],
\end{align}
via term-by-term comparison, which in turn implies
\begin{equation}\label{g x<0 2}
\xi^4 g(2h\xi)^2 \leq (2h \xi)^2  \cdot \xi^4 [1+g(2h\xi)]^2 \leq  (2h \xi)^2 h^{-2} [f_h(\xi)+ h]^2 .
\end{equation}
In this way, using \eqref{eugh} we deduce that
\begin{align}\label{872<0}
\int_{-\infty}^0 \frac{\xi^4 g(2h\xi)^2\,d\xi}{|f_h(\xi)/h-k^2|^2|\xi^2 - k^2|^2}
	\lesssim_k \int_{-\infty}^0 \frac{(2h \xi)^2 \,d\xi}{(\xi^2 + 1)^2} \lesssim  h^2.
\end{align}

Combining \eqref{872'}, \eqref{872>0}, and \eqref{872<0} completes the proof of \eqref{872}.
\end{proof}

In Proposition~\ref{P:phi}, the solutions $\phi_\pm$ were constructed for general depth $h$.  We recast this in the vocabulary suitable for consideration of the shallow-water limit as follows:  If $\tfrac1h f_h(\zeta)\notin \sigma(\breve L_u)$, then
\begin{equation}\label{breve phi}
\breve \phi_\pm(x;\zeta,u) := \phi_\pm(x;\zeta,hu)  \qtq{satisfies} \breve \phi_\pm = e^{i\zeta x} + \breve  R_0(\zeta) u \breve \phi_\pm
\end{equation}
and is the unique such solution with $\breve \phi_\pm - e^{i\zeta x} \in L^2(\R)$.  Our next proposition describes the $h\to0$ limit of these solutions.

\begin{prop}\label{P:KdV}
Fix $u\in C^\infty_c(\R)$ and $k\in \C$ with  $\pm\Im k >0$ and $k^2\notin\sigma(L_u^\textsc{kdv})$.  Then, $\tfrac1h f_h(\zeta_\pm)\notin \sigma(\breve L_u)$ for $h$ sufficiently small and
\begin{align}\label{903}
\breve \phi_\pm(x;\zeta_\pm,u) &- e^{i\zeta_\pm x} \to R_u^\textsc{kdv}(k) e^{ikx}u \quad\text{in $L^2_x(\R)$ sense as $h\to 0$.}
\end{align}
For the definitions of $\zeta_\pm$ and $\breve\phi_\pm$, see \eqref{zeta of k} and \eqref{breve phi}, respectively.
\end{prop}

\begin{proof}
As $k^2\notin\sigma(L_u^\textsc{kdv})$, the operator $1 - R_0^\textsc{kdv}(k) u$ is invertible (cf. \eqref{404}). By Lemma~\ref{L:KdV}, it then follows that $1 - \breve R_0(\zeta_\pm) u $ is invertible and so $\tfrac1h f_h(\zeta_\pm)\notin \sigma(\breve L_u)$ for $h$ sufficiently small.

Now choose $\chi\in C^\infty_c(\R)$ such that $\chi\equiv 1$ on $\supp (u)$ and let $\chi_k(x):=e^{ik x}\chi(x)$.  From \eqref{405}, we have
\begin{align*}
\breve\phi_\pm(x;\zeta_\pm,u) - e^{i\zeta_\pm x}  = [ 1 - \breve R_0(\zeta_\pm) u  ]^{-1} \breve  R_0(\zeta_\pm) u \, \chi e^{i\zeta_\pm x} .
\end{align*}
An application of Lemma~\ref{L:KdV} then yields
\begin{align*}
\breve\phi_\pm(x;\zeta_\pm,u) - e^{i\zeta_\pm x}  \to [ 1 - R_0^\textsc{kdv}(k) u  ]^{-1} R_0^\textsc{kdv}(k) u \, \chi_k \quad\text{in $L_x^2(\R)$ as $h\to 0$.}
\end{align*}
Claim \eqref{903} then follows from the resolvent identity and $[u\chi_k](x)=e^{ik x}u(x)$.
\end{proof}

\begin{theorem}\label{T:KdV}
Fix $u\in C^\infty_c(\R)$ and $\pm\Im k >0$ with $k^2\notin\sigma(L_u^\textsc{kdv})$.  Then
\begin{gather}\label{904a}
\mathfrak t(\zeta_\pm;hu) \to \mathfrak t^\textsc{kdv}(k), \\
	\label{904b}
\frac{\det_2\bigl(1 - R_0(\zeta\mp\frac{i\pi}h) q\bigr)}{\det_2\bigl(1- R_0(\zeta) q\bigr)} \exp\biggl\{\pm \tfrac{i}{1-e^{-2h\zeta}}\int_\R q(x)\, dx \biggr\}
	 \to {\frac1{\det\bigl(1- R_0^\textsc{kdv}(k) u\bigr)}},
\end{gather}
as $h\to 0$. In this way, Theorem~\ref{T:main} reproduces \eqref{E:JP} in the limit $h\to0$. See \eqref{zeta of k} for the definition of $\zeta_\pm= \zeta_\pm(k;h)$.
\end{theorem}

\begin{proof}
By Proposition~\ref{P:KdV},  $\tfrac1h f_h(\zeta_\pm)\notin \sigma(\breve L_u)$ for $h$ sufficiently small.  This guarantees the applicability of Proposition~\ref{P:limit} to construct $\mathfrak t(\zeta_\pm;hu)$.  In particular,
\begin{align}\label{5:a}
\mathfrak t(\zeta_\pm;hu)  = 1 \pm \tfrac{ih}{1-e^{-2h\zeta_\pm}}
	\biggl[\int_\R u(y)\,dy + \int_\R e^{-iy\zeta_\pm} u(y) \bigl[\breve \phi_\pm(y) - e^{iy\zeta_\pm}\bigr]\, dy \biggr] .
\end{align}

Applying \eqref{zeta of k'} and \eqref{903} in \eqref{5:a}, we deduce that
\begin{align*}
\mathfrak t(\zeta_\pm;hu) \to 1 \pm \tfrac{i}{2k}\int_\R u(y)\,dy \pm \tfrac{i}{2k} \bigl\langle e^{iky} u(y), R_u^\textsc{kdv}(k) e^{iky}u\bigr\rangle \qtq{as} h\to0.
\end{align*}
Recalling that 
\begin{align*}
\phi_\pm^\textsc{kdv}(y) = e^{iky} + R_u^\textsc{kdv}(k)e^{iky} u
\end{align*}
is the unique solution to \eqref{KdV phi}, we arrive at
\begin{align}\label{904''}
\mathfrak t(\zeta_\pm;hu) \to 1\pm \tfrac{i}{2k} \int_\R e^{-iky} u(y)\phi_\pm^\textsc{kdv}(y)\,dy  \qtq{as} h\to0.
\end{align}
Comparing with \eqref{t KdV}, we obtain \eqref{904a}.

From \eqref{u defn} and \eqref{yay}, we see that LHS\eqref{904b} can also be written
\begin{equation}\label{t as det}  
\frac{\det_2\bigl(1 - \breve R_0(\zeta_\pm\mp\frac{i\pi}h) u\bigr)}{\det_2\bigl(1- \breve R_0(\zeta_\pm) u\bigr)}
	\exp\biggl\{\pm \tfrac{ih}{1-e^{-2h\zeta_\pm}}\int_\R u(x)\, dx \biggr\} .
\end{equation}

Applying \eqref{zeta of k'}, we obtain
\begin{align*}
\exp\biggl\{\pm \tfrac{ih}{1-e^{-2h\zeta_\pm}}\int_\R u(x)\, dx \biggr\} \to \exp\biggl\{\pm \tfrac{i}{2k}\int_\R u(x)\, dx \biggr\} \qtq{as} h\to0.
\end{align*}

By virtue of Lemma~\ref{L:KdV}, the regularized determinant in the denominator of \eqref{t as det} converges to $\det_2(1-R_0^\textsc{kdv}(k) u)$ as $h\to0$; moreover, this is non-zero because $k^2\notin\sigma(L_u^\textsc{kdv})$.
  
The numerator in \eqref{t as det} is easily seen to converge to unity as $h\to 0$: From \eqref{f symm} and \eqref{zeta of k} we have $f_h(\zeta_\pm \mp \tfrac{i\pi}h) = h k^2 \mp \tfrac{i\pi}{h}$.  Thus by applying \eqref{yay} and Lemma~\ref{L:Rel HS}, we find
\begin{align*}
\bigl\| \breve R_0(\zeta\mp\tfrac{i\pi}h) u \bigr\|_{\mathfrak I_2}^2  \lesssim \| h u\|_{L^2}^2 	
	\frac{ 1 + h^{-1/2} \dist\bigl(hk^2\mp \tfrac{i\pi}{h},[0,\infty)\bigr)^{-1/2} }{\dist\bigl(hk^2\mp \tfrac{i\pi}{h},[0,\infty)\bigr)} = O( h^{3}).
\end{align*}

Using these observations in \eqref{t as det}, we arrive at
\begin{equation} 
\text{LHS\eqref{904b}} \to \frac{1}{\det_2\bigl(1-R_0^\textsc{kdv}(k) u\bigr)}
	\exp\biggl\{\pm \tfrac{i}{2k}\int_\R u(x)\, dx \biggr\} \quad\text{as $h\to 0$.}  \label{t as det'} 
\end{equation}
It remains only to identify RHS\eqref{t as det'} as RHS\eqref{904b}.

In order for RHS\eqref{904b} to even make sense, we need $R_0(k)u$ to be a trace class operator.  This follows from the Birman--Solomjak criterion; see \cite[Theorem~4.5]{MR2154153}.  Computing the trace is quite elementary: the operator $R_0(k)u$ has explicit integral kernel
\begin{align*}
 \pm\tfrac{i}{2k} e^{\pm ik|x-y|} u(y) \qtq{and so} \tr\bigl( R_0^\textsc{kdv}(k)u \bigr) =  \pm \tfrac{i}{2k}  \smash[t]\int u(x)\,dx. 
\end{align*}
In this way, we see that
\begin{align*}
\det_2\bigl(1-R_0^\textsc{kdv}(k) u\bigr) = \det \bigl(1-R_0^\textsc{kdv}(k) u\bigr) \exp\biggl\{\pm \tfrac{i}{2k}\int_\R u(x)\, dx \biggr\} 
\end{align*}
and so \eqref{904b} is proved.\end{proof}


\end{document}